\documentclass[11pt]{article}
\usepackage[titletoc]{appendix}
\usepackage[labelsep=period]{caption}
\usepackage{bm}
\usepackage{fullpage}
\usepackage{epsfig}
\usepackage{graphics}
\usepackage{latexsym}
\usepackage{mathrsfs}
\usepackage{pifont}
\usepackage{yhmath}
\usepackage{bbm}
\usepackage{dsfont}
\usepackage{eufrak}
\usepackage{wasysym}
\usepackage{underscore}
\usepackage{epstopdf}
\usepackage{fontenc}
\usepackage{mathrsfs}
\usepackage{amsthm}
\usepackage{amssymb}
\usepackage{latexsym}
\usepackage{amsmath,amsfonts}
\usepackage{mathrsfs}
\usepackage{cases}
\usepackage{latexsym,bm}
\usepackage{indentfirst}
\usepackage{color}
\usepackage{ifpdf}
\usepackage{graphicx}
\usepackage{psfrag}
\usepackage{enumerate}
\pdfoutput=1
\usepackage[
pdfauthor={CSY},
pdftitle={Lower bound of the average degree of a $\Delta$-critical graph},
pdfstartview=XYZ,
bookmarks=true,
colorlinks=true,
linkcolor=blue,
urlcolor=blue,
citecolor=blue,
bookmarks=true,
linktocpage=true,
hyperindex=true
]{hyperref}

\usepackage{graphicx}
\usepackage{color}

\title{Ramsey-Tur\'an numbers
 for intersecting odd cliques}
\author{Guantao Chen\thanks{ Supported in part by NSF grant DMS-1855716}    ~and
 Min Liu\thanks{Supported in part by the Research Foundation of College of Economics, Northwest  University of Politics and Law (No. 19XYKY07),  and  by China Scholarship Council (No. 201808610115). }
 \unskip\\[.5em]
{\small  Department of Mathematics and Statistics, Georgia State University, Atlanta, GA 30303}\\
{\small  College of Economics, Northwest University of Politics and Law, Xi'an, Shaanxi, China 710063}\\
}

\date{}

\newtheorem{lem}{Lemma}
\newtheorem{thm}{Theorem}

\numberwithin{equation}{section}

\newcommand{\FF}{\mathcal{F}}

\begin{document}
\newcommand{\udots}{\mathinner{\mskip1mu\raise1pt\vbox{\kern7pt\hbox{.}}
\mskip2mu\raise4pt\hbox{.}\mskip2mu\raise7pt\hbox{.}\mskip1mu}}

\maketitle

\begin{abstract}
Given a graph  $H$  and a function  $f:\mathbb{Z}^+ \longrightarrow \mathbb{Z}^+ $, the Ramsey-Tur\'an number of $H$ and $f$, denoted by $RT(n, H, f(n))$, is the maximum number of edges a graph $G$ on $n$ vertices can have, which does not contain $H$ as a subgraph and also does not contain a set of $f(n)$ independent vertices.  Let $r$ be a positive integer. In 1969, Erd\H{o}s and S\'os  proved that  $RT(n,K_{2r+1},o(n))=\frac{n^2}{2}(1-\frac{1}{r})+o(n^2)$.  Let $F_k(2r+1)$ denote the graph consisting of $k$ copies of complete graphs $K_{2r+1}$ sharing exactly one vertex.  In this paper, we show that $RT(n,F_k(2r+1),o(n))=\frac{n^2}{2}(1-\frac{1}{r})+o(n^2)$, which is of the same magnitude with $RT(n, K_{2r+1}, o(n))$.

\end{abstract}


\vskip 0.2in \baselineskip 0.18in
\section{Introduction }
Let $H$ be a graph. A graph $G$ is called {\em $H$-free} if $G$ does not contain  a copy of $H$ as a subgraph.  For a positive integer $n$, let $ex(n, H)$ denote the maximum number of edges of an $H$-free graph of order $n$, and call such a graph an {\em extremal graph} for $H$.  For any positive integer $r$, let $K_r$ denote the $n$-vertex complete graph.  Tur\'an's classical result states that  the balanced complete $n$-vertex $r$-partite graph, denoted by $T_r(n)$, is the unique
$K_{r+1}$-extremal graph of order $n$.  A graph on $2k+1$ vertices consisting of $k$ triangles which intersect in exactly one common vertex is called a {\em $k$-fan} and denoted by $F_k$. Erd\H{o}s, F\"uredi, Gould, and Gunderson~\cite{ErdosFGG95} determined  $ex(n, F_k)$ for $n\ge 50k^2$. A graph on $(r-1)k+1$ vertices consisting of $k$ copies of $K_r$ sharing exactly one common vertex is called a {\em $(k, r)$-fan} and denoted by $F_k(r)$. Chen, Gould, Pfender, and Wei~\cite{ChenGPW03} determined $ex(n, F_k(r))$ for $n \ge 16k^3r^8$.

Notice that the independence number $\alpha(T_r(n))$ is at least $\frac{n}{r}$.
Given a graph $H$ and a function $f: \ \mathbb{Z}^+ \longrightarrow \mathbb{Z}^+$, the {\em Ramsey-Tur\'an number} for a large positive integer $n$ with regard to $H$ and $f(n)$,  denoted by $RT(n, H, f(n))$, is the maximum number of edges of an $n$-vertex
$H$-free graph $G$ with $\alpha(G) \le f(n)$. Erd\H{o}s and S\'{o}s in 1969~\cite{ErdosSos69} studied of Ramsey-Tur\'an numbers and showed that $RT(n, K_{2r+1}, o(n)) = \frac {n^2}2 (1 -\frac 1r) +o(n^2)$, which is approximately the value of $ex(n, K_{r+1})$. The study of  Ramsey-Tur\'an numbers
 has drawn a great deal of attention over the last 40 years; see the survey by Simonovits and S\'os \cite{survey} and  \cite{1983moreresults, Fox2015, liuhong} for more results on various Ramsey-Tur\'{a}n problems.  Inspired by the study of extremal numbers $ex(n, F_k)$ and $ex(n, F_k(r))$ in~\cite{ErdosFGG95} and~\cite{ChenGPW03}, respectively, we consider Ramsey-Tur\'an numbers for $F_k(r)$ and obtained the following result.

 \begin{thm}\label{kr}
For any two fixed positive integers $k$ and $r$,
\begin{equation*}
{\displaystyle RT(n, F_{k}(2r+1),o(n))=\frac{n^2}{2}\left(1-\frac{1}{r}\right)+o(n^2).  }
\end{equation*}
\end{thm}

In the next section we will first give a short proof of the case $r=1$, i.e., $RT(n, F_{k}(3), o(n)) = o(n^2)$, and then prove the general case.  We now introduce notations and terminology  that will be used in the proof.

All graphs considered in this paper are simple graphs. Let  $G=(V(G), E(G))$  be a graph with vertex set $V(G)$ and edge set $E(G)$.  Denote by $|G|$ and $||G||$ for $|V(G)|$ and $|E(G)|$, respectively;  and call them the order and the size of $G$, respectively.  Denote by $G(n)$ a graph of order $n$ and denote by $G(n, m)$ a graph of order $n$ and size $m$. For a vertex $x\in V(G)$, the \emph{neighborhood} of $x$ in $G$, denoted by $N_G(x)$,  is the set of vertices adjacent to $x$, i.e.,  $N_{G}(x)=\{y\in V(G):xy\in E(G)\}$. We use $N(x)$ for $N_G(x)$ when the referred graph $G$ is clear.  The \emph{degree} of $x$ in $G$, denoted by $d_G(x)$, or $d(x)$, is the cardinality of $N_{G}(x)$, i.e., $d_G(x) = |N_G(x)|$.  Denote by $\delta(G)$ the minimum degree in $G$. For a subset $X\subseteq V(G)$, let $G[X]$ denote the subgraph of $G$ induced by $X$, i.e., a subgraph of $G$ with vertex set $X$ and the set of all edges with two ends in $X$.
For a vertex set $X = \{x_1, x_2, \dots, x_r\}$,
denote by  $G[x_1,x_2,\dots, x_r]$ for  $G[\{x_1,x_2,\dots, x_r\}]$.

The
\emph{matching number} of a graph $G$, denoted by $\nu(G)$, is the maximum number of edges in a matching in $G$. Let $\alpha(G)$ denote the {\em independence number} of $G$. Clearly, if $G$ is an $n$-vertex graph, then $\alpha(G)\ge n - 2\nu(G)$.

\section{Proof of Theorem~\ref{kr} }

\subsection{Prilimilary}
We first consider triangles sharing a common vertex and prove the following simple fact.

\begin{lem}\label{lem 1}
For any  fixed positive integer $k$, $RT(n,F_{k}(3), o(n))=o(n^2)$.
\end{lem}

\begin{proof}
Let $G$ be an $n$-vertex $F_{k}(3)$-free graph with independence number
$\alpha(G) \le \epsilon n$ for some small positive real number $\epsilon$.
  For any $x\in V(G)$,
since $G$ does not contain $k$ edge-disjoint triangles intersecting in the vertex $x$, the neighborhood of $x$ contains at most $k-1$ independent edges, i.e., $\nu(G[N(x)]) \leq k-1$.  Thus, $\alpha (G[N(x)])\geq d(x)-2(k-1)$. Since $\alpha(G[N(x)]) \le \alpha(G) \le \epsilon n$, we have $d(x)\leq \epsilon n+2(k-1)$. Counting the total degree of $G$, we have the following inequality:
$|E(G)|\leq \frac 12 (\epsilon +2(k-1)/n)n^2$. Since $\lim_{n \rightarrow \infty} \frac 12 (\epsilon +2(k-1)/n) = \frac 12 \epsilon$, we complete the proof of Lemma~\ref{lem 1}.
\end{proof}

The following result from~\cite{ErdosSos69} is needed in our proof and for completeness we give its proof here.

\begin{lem} [Erd\H{o}s and  S\'os] \label{lem 2}
Let $\beta$ and $\epsilon$ be two positive numbers with $0< \beta < \frac{1}{2}$ and $0 < \epsilon < 1$. Then for every positive number $c$ with $0<c < \sqrt{\beta\epsilon}$ and any graph  $G(n)$ with $|E(G)|\geq  \beta n^2 (1 + \epsilon)$,  there is a subgraph $G(m)\subseteq G(n)$ with $m >  c n$ and $\delta(G(m)) > 2 \beta(1 + \epsilon/2)m$.
\end{lem}

\begin{proof}
To avoid cumbersome notation, we assume without loss of generality that $c n$ is an integer. Suppose on the contrary a such graph $G(m)$ does not exist.  Denote the graph $G(n)$ by $G$ with $|E(G)|\geq \beta n^2 (1 + \epsilon)$.  Then the vertices of $G$ can be written in a sequence $x_1, x_2, \dots, x_n$ so that for every $i \leq (1-c)n$ the degree of $x_i$ in $G[x_i, x_{i+1}, \dots, x_n]$ is less than $2\beta (1+ \epsilon/2) (n-i+1)    $.

\begin{eqnarray*}
\beta n^2(1+\epsilon)  \leq |E(G)|   & \leq & 2\beta \left(1+\frac{\epsilon}{2}\right)(n+(n-1)+\dots +(cn+1))+{cn \choose 2}\\
& < & 2\beta \left(1+\frac{\epsilon}{2}\right)\left(\frac{n^2}{2}\right)+\frac{c^2n^2}{2}\\
& < & \beta n^2(1+\epsilon) \qquad \mbox{ (since $c < \sqrt{\beta \epsilon}$), }
\end{eqnarray*}
which gives  a contradiction.
\end{proof}

\begin{lem}\label{lem 3}
Let $r$ be a positive integer. For any  $\epsilon >0$, there exists positive integer $N=N(\epsilon)$ such that for a clique $D$ of a graph $G(n)$ with $n>N$ and $|D| \le r$,  if
$d(v) \ge \left(1-\frac{1}{r}+\frac{\epsilon}{3}\right)n$ for each $v\in D$,
then there exists a clique $D^*$ that contains $D$ as a proper subset.
\end{lem}

\begin{proof}
Let $m$ denote the number of vertices in $V(G) - D$ that are adjacent to all vertices in $D$. It is sufficient to show that $m > 0$.  By counting the edges between $D$ and $G-D$, we have the following inequality.

\begin{eqnarray*}
|D|m+(|D|-1)(n-|D|-m)
 \geq \sum_{v\in D} (d(v)-(|D|-1)) \ge |D|\left(\left(1 -\frac 1r +\frac {\epsilon}3\right)n -(|D|-1)\right)
\end{eqnarray*}

\noindent Thus, $m   \geq  |D|\left(1-\frac{1}{r}+\frac{\epsilon}{3}\right)n-(|D|-1)n
=   n-\left(\frac{1}{r}-\frac{\epsilon}{3}\right)|D|n >  0$, where in the last inequality we used
   $|D|\leq r$.
\end{proof}

\begin{lem}\label{lem 4}
Let $r$ be a positive integer. For any  $\epsilon >0$,   there exist $\delta =\frac{\epsilon}{4}$ and $N=N(\epsilon)$ such that for any clique $D$
of a graph  $G(n)$ with $n>N$ and $\alpha(G(n))\leq \delta n$, if $|D|\leq 2r$
and $d(v) \geq (1-\frac{1}{r}+\frac{\epsilon}{3})n$ for each $v\in D$,  then there exists a clique $D^*$ with $|D^*|=|D|+1$ such that $|D^*\cap D|\ge |D| -1$.
\end{lem}

\begin{proof}
We set $|D|=s\leq 2r$ and $G = G(n)$.  If there is a vertex $v\in V(G) -D$ that is adjacent to
all vertices in $D$, then $D\cup \{v\}$ is the desired clique. So, we assume that
$d_D(v) \le s-1$ for all $v\in V(G) -D$.  We divide $V(G- D)$ into two vertex-disjoint sets $U$ and $W$ such that
\begin{eqnarray*}
U & = & \{v\in V(G) - D \ |\ d_D(v)\leq s-2\}, \mbox{ and} \\
W & = & \{v\in V(G) -  D) \ |\ d_D(v)=s-1\}.
\end{eqnarray*}

\noindent Computing  the edges between $D$ and $V(G) - D$, we get the following inequality.
\begin{eqnarray*}
(s-2)|U|+(s-1)|W| \geq \sum_{v\in D} d_{G-D}(v) \ge s\left(\left(1-\frac 1r + \frac {\epsilon}3\right)n-(s-1)\right)
\end{eqnarray*}

\noindent By assuming $n \ge N > \frac 3{\epsilon}$ and applying the inequality $s \le 2r$, we get
\[
|W|\geq \left(2-\frac{s}{r}+s\frac{\epsilon}{3}-\frac{s}{n}\right)n
>  \left(s\frac{\epsilon}{3}-\frac{s}{n}\right)n
> 0.
\]

\noindent Let $D= \{v_1,v_2,\dots,v_s\}$ and divide $W$ into $s$ vertex-disjoint sets as $W_1,W_2,\dots W_s$ such that $W_i=\{w\in W \ |\ v_iw\notin E(G)\}$
for each $1\le i \le s$.  We claim $W_i$ is an independent set. Otherwise, let $x$ and $y$ be two adjacent vertices in $W_i$. Then, $D^* = D\backslash \{v_i\}\cup\{x, y\}$ is the desired clique.  Hence,

\[
\alpha (G)\geq max\{|W_i|,{1\leq i\leq s}\}> \left(\frac{\epsilon}{3}-\frac{1}{n}\right)n>\delta n,
\]
giving a contradiction.
\end{proof}

\subsection{Proof of Theorem~\ref{kr}}
Recall the graph $F_{k}(2r+1)$ is the union of $k$ copies of $K_{2r+1}$ sharing a common vertex. We call the common vertex the center of $F_{k}(2r+1)$.  For convention, if $k=0$, $F_{k}(2r+1)$ is a single vertex graph, which is the center.  For any $m$ nonnegative integers
$k_1$, $k_2$, $\dots$, $k_m$, let $F_{k_1, k_2, \dots, k_m}(2r+1)$ be the graph consisting of
a clique $B$ with $m$ vertices and $m$ vertex-disjoint copies of $F_{k_i}(2r+1)$ with its center in $B$.  In the above definition, we call $B$ the {\it base}  of $F_{k_1, k_2, \dots, k_m}(2r+1)$.

We now show
$RT(n,F_{k}(2r+1),o(n))=\frac{n^2}{2}(1-\frac{1}{r})+o(n^2)$, where $r\geq 2$, as the theorem holds when $r=1$ by Lemma \ref{lem 1}. Since $RT(n, K_{2r+1}, o(n)) = \frac{n^2}{2} (1 -\frac{1}{r}) + o(n^2)$,  we only need to show that $RT(n,F_{k}(2r+1),o(n)) \le \frac{n^2}{2}(1-\frac{1}{r})+o(n^2)$.   Let $\epsilon>0$ be a small positive real number. We will show that there exists a positive number $\delta := \delta(\epsilon)$ and a large positive integer $N:=N(\epsilon)$ such that  all graphs $G(n)$ with $n>N$, if $|E(G)|\geq (1-\frac{1}{r}+\epsilon)n^2/2$ and $\alpha(G(n))\leq \delta n$, then there exists a $F_k(2r+1)$ as a subgraph of $G(n)$.  In the proof, we will take $\delta =\frac{\sqrt{2}}{10}\epsilon^2$ and we will not specify the value of $N(\epsilon)$ by only assuming $n$ is large enough to ensure our counting work.

Since $|E(G)| \ge \left(1-\frac{1}{r}+\epsilon\right)n^2/2$, by Lemma~\ref{lem 2}, $G$ contains a subgraph $G_1$ of order $n_1 \ge \sqrt{\epsilon/2}n$ such that the minimum degree $\delta(G_1) \ge (1 -\frac 1r - \frac{\epsilon}2) n_1$. Since $\alpha(G_1) \le \alpha(G) \le \frac{\sqrt{2}}{10} \epsilon^2 n \le \frac{\epsilon}5 n_1$.  So, we assume without loss of generality that $\delta(G) \ge (1 -\frac 1r - \frac{\epsilon}2) n$ and $\alpha(G) \le \frac{\sqrt{\epsilon}}{5} n$.

For each integer $m$ with $1 \le m \le r+1$, let $\FF_{m}$ be the set of $m$-tuples  $(k_1, k_2, \dots, k_m)$ of $m$ nonnegative integers $k_1$, $k_2$, $\dots$, $k_m$ listed with non-increasing order such that $G$ contains a copy of $F_{k_1, k_2, \dots, k_m}(2r+1)$.  By Lemma~\ref{lem 3}, $G$ contains a copy of $K_{m}$ for every $1\le m \le r+1$. Thus,
$\FF_m \ne \emptyset$.  We also notice that if there exists an $m$-tuple $(k_1, k_2, \dots, k_m)\in \FF_m$ with $k_1 \ge k$, then $G$ contains a copy of $F_k(2r+1)$. So, we assume $k_1 \le k-1$ for all $(k_1, k_2, \dots, k_m)\in \FF_m$.  Consequently, $|\FF_m| \le m^k$.

For each $\FF_m$, we define a lexicographic order $\prec$ such that
$(k_1, k_2, \dots, k_m) \prec (\ell_1, \ell_2, \dots, \ell_m)$ if there exist $s$ with $1\le s\le m$ such that
$k_i = \ell_i$ for $i< s$ and $k_i < \ell_i$.  Clearly, $(0,0, \dots, 0)$ is the minimum element in $\FF_m$ following this order.

Let $(k_1, k_2, \dots, k_{r+1})$ be the maximum element in
$\FF_{r+1}$. We will lead a contradiction by showing that either $G$ contains a copy of $F_{k}(2r+1)$ or $(k_1, k_2, \dots, k_{r+1})$ is not maximum in $\FF_{r+1}$. Let $F$ be a copy of $F_{k_1, k_2, \dots, k_{r+1}}(2r+1)$ in $G$ and $B$ be the base. Denote $B$ by $\{v_1, v_2, \dots, v_{r+1}\}$  and assume $v_i$ is the center of the corresponding  $F_{k_i}(2r+1)$ in $F$.  Let $H$ be obtained from $G_1$ by deleting all vertices in $V(F) -B$, i.e., $H = G - (V(F) -B)$. Since $|V(F)| \le 2r(k-1)(r+1)+(r+1)$ and $\delta(G) \ge (1 -\frac 1r +\frac{\epsilon}2)n$,  we may assume the minimum degree $\delta(H) \ge (1-\frac 1r - \frac{\epsilon}3) |H|$ and $\alpha(H) \le \frac{\epsilon}4 |H|$  provided $n$ is large enough.

Starting with clique $B_0 = B$ in $H$, we apply Lemma~\ref{lem 4} repeatedly $r$ times, we get  a sequence of cliques $B_1$, $B_2$, $\dots$, $B_r$ such that
$|B_i| = |B_{i-1}| +1$ and $|B_i\cap B_{i-1}| \ge |B_{i-1}| -1$.  At the end, we have
$|B_r| = |B_0| +r = (r+1) +r = 2r+1$ and $|B_r \cap B_0| \ge |B_0| -r \ge 1$.  Let $s\ge 1$ be the smallest integer such that $v_s\in B_r\cap B_0$. Then, we get subgraph $F^*$ which is a copy of $F_{k_1, k_2, \dots, k_s+1}(2r+1)$ with base $B^*=\{v_1, v_2, \dots, v_s\}$. Let $H^* = G- (V(F^*) - B^*)$.

Applying Lemma~\ref{lem 3} repeatedy $(r+1 -s)$ times, we get a clique $B' \supseteq B^*$ with $r+1$ vertices in $H^*$. Combining $B'$ with $F^*$, we get a copy of $F_{k_1, \dots, k_s+1, 0,\dots, 0}(2r+1)$ with base $B'$.  Let $\ell_1, \ell_2, \dots, \ell_s$ be an new list  of
$k_1, k_2, \dots, k_s+1$ in non-increasing order. Then, $(r+1)$-tuple $(\ell_1, \ell_2, \dots, \ell_s, 0, \dots, 0)\in \FF_{r+1}$. But,
$(k_1, k_2, \dots, k_{r+1}) \prec (\ell_1, \ell_2, \dots, \ell_s, 0, \dots, 0)$, giving a contradiction to the maximality of $(k_1, k_2, \dots, k_{r+1})$.

\section{Acknowledgement}

We would like to thank Yan Cao for the helpful communication.

\bibliographystyle{plain}
\bibliography{CL-Reference}
\bibliographystyle{amsplain}

\end{document}